%% file: main.tex
\documentclass[onecolumn,longbibliography]{siamart0216}
\usepackage{comment}
\usepackage{amsmath}
\usepackage[makeroom]{cancel}
\usepackage[numbers,sort&compress]{natbib}

\input{ex_shared}

\ifpdf
\hypersetup{
  pdftitle={\TheTitle},
  pdfauthor={\TheAuthors}
}
\fi

\begin{document}

\maketitle

\begin{abstract}
 We present an algorithm to  compute the Jordan chain of a nearly defective matrix with a $2\times2$ Jordan block.  
 The algorithm is based on  an inverse-iteration procedure and only needs information about the invariant subspace corresponding to the Jordan chain, making it suitable for use with large matrices arising in applications, in contrast with existing algorithms which rely on an SVD. The algorithm produces the  eigenvector and Jordan vector with $O(\eps)$ error, with $\eps$ being the distance of the given matrix to an exactly defective matrix.    As an example, we demonstrate the use of this algorithm in a problem arising from electromagnetism, in which the matrix has size $212^2\times 212^2$.  An extension of this algorithm is also presented which can achieve higher order convergence [$O(\eps^2)$] when the matrix derivative is known.
  \end{abstract}

\section{Introduction}
\label{intro-section}
\input{introduction.tex}

\section{Finding the invariant subspace}
\label{subspace-section}
\input{subspace.tex}

\section{Computing the Jordan chain}
\label{chain-section}
\input{chain.tex}

\section{Implementation}
\label{numerical-section}

\input{numerical_results.tex}

\section{Concluding remarks}
\label{conclusions-section}

In this paper, we presented efficient algorithms for computing Jordan chains of large matrices with a $2\times2$ Jordan block. These tools can be readily applied to  study the limiting behavior of physical systems near second-order exceptional points. 
The presented algorithm can be easily  generalized to handle larger Jordan  blocks by first finding an $N$-dimensional degenerate subspace and a reduced $N\times N$ eigenproblem, and then computing the Jordan chain of its nearest $N\times N$ defective matrix. 
Since $N$ (the algebraic multiplicity of the eigenvalue) will in practice be very small, existing SVD-based methods~\cite{leung1999null,mailybaev2006} can be applied to find this nearby defective $N \times N$ matrix.
Moreover,  defective  nonlinear eigenproblems can be handled using a similar approach with  proper  modification of the chain relations at the EP~\cite{szyld2013several}.

\section*{Appendix}
\label{adjoint-section}
\input{adjoint.tex}

\section*{Acknoledgement}
The authors would like to thank Gil Strang and Alan Edelman  for helpful discussions. This work was partially supported by the Army Research Office through the Institute for Soldier Nanotechnologies under Contract \\No.~W911NF-13-D-0001.


\end{document}

%% file: ex_shared.tex
\usepackage{lipsum}
\usepackage{amsfonts}
\usepackage{graphicx}
\usepackage{epstopdf}
\usepackage{mathrsfs}
\usepackage{graphicx, pstricks}
\usepackage{algorithm, algorithmicx}
\usepackage[noend]{algpseudocode}

\makeatletter
\def\input@path{{../figures/}{../bibliography/}}
\makeatother

\AtBeginDocument{%
  \paperwidth=\dimexpr
    1in + \oddsidemargin
    + \textwidth
    + 1in + \oddsidemargin
  \relax
  \paperheight=\dimexpr
    1in + \topmargin
    + \headheight + \headsep
    + \textheight
    + 1in + \topmargin
  \relax
\usepackage[pass]{geometry}\relax
}

\ifpdf
  \DeclareGraphicsExtensions{.eps,.pdf,.png,.jpg}
\else
  \DeclareGraphicsExtensions{.eps}
\fi

\newcommand{\eps}{{\varepsilon}}

\renewcommand{\eqref}[1]{Eq.~(\ref{eq:#1})}

\DeclareMathOperator{\trace}{trace}

\newtheorem{thm}{Theorem}

\newtheorem{propn}{Proposition}

\newcommand{\TheTitle}{Scalable computation of  Jordan chains} 
\newcommand{\TheAuthors}{F. Hern{\'a}ndez, A. Pick, and S. G. Johnson}

\headers{\TheTitle}{\TheAuthors}

\title{{\TheTitle}}

\author{
  Felipe Hern{\'a}ndez\thanks{Department of Mathematics, Massachusetts Institute of Technology, 77 Massachusetts Avenue, Cambridge, Massachusetts 02139, USA}
  \and
  Adi Pick\thanks{Department of Physics, Harvard University, Cambridge, Massachusetts 02138, USA (adipick@physics.harvard.edu)}
  \and
  Steven G.  Johnson$^*$}

\usepackage{amsopn}


%% file: introduction.tex
In this paper, we present algorithms to find approximate Jordan vectors of nearly defective matrices, designed to be scalable to large sparse/ structured matrices (e.g., arising from discretized partial differential equations),  by requiring only a small number of linear systems to be solved. We are motivated by a recent explosion of interest in the physics literature in ``exceptional points'' (EPs), in which a linear operator $A(p)$ that depends on some parameter(s) $p$ becomes  defective  at $p = p_0$, almost always with a $2\times2$ Jordan block~\cite{lin2011unidirectional,peng2014parity,feng2013experimental,doppler2016dynamically,xu2016topological,brandstetter2014reversing,zhen2015}.  A variety of interesting phenomena arise in the vicinity of the EP.  The limiting behavior at the EP~\cite{mailybaev2005geometric,Pick2016}, as well as perturbation theory around it~\cite{Seyranian2003}, is understood in terms of the Jordan chain relating an eigenvector $x_0$ and a Jordan vector $j_0$:
\begin{align}
A_0 x_0 &= \lambda_0 x_0, \\
A_0 j_0 &= \lambda_0 j_0 + x_0,
\end{align}
where $A_0 = A(p_0)$.
Typically in   problems that arise from  modeling  realistic physical  systems, one does not know $A_0$ precisely, but only  has  a good approximation $\mu$ for the degenerate eigenvalue $\lambda_0$ and 
 a matrix $A_\eps = A(p_0 + \eps)$ \emph{near} to $A_0$, limited both by numerical errors and the difficulty of tuning $p$ to find the EP exactly.  Given $A_\eps$ and $\mu$, the challenge is to approximately compute $x_0$, $\lambda_0$, and $j_0$ to at least $O(\eps)$ accuracy.

Existing algorithms to approximate the Jordan chain~\cite{leung1999null,mailybaev2006} rely on computing the dense SVD of $A_\eps$, which is not scalable, or have other problems described below.  Instead, we want an algorithm that only requires a scalable method to solve linear systems $(A_\eps - \mu \mathbb{I}) x = y$ (``linear solves''), and such algorithms are readily  available (e.g. they are needed anyway to find the eigenvalues away from the EP that are being forced to coincide).  Using only such linear solves, we show that we can compute the Jordan chain to $\mathcal{O}(\eps)$ accuracy (Sec.~\ref{chain-section}), or even $\mathcal{O}(\eps^2)$ accuracy if $dA/dp$ is also supplied (Sec.~\ref{adjoint-section}), which we find is similar to the previous dense algorithms (whose accuracy had not been analyzed)~\cite{mailybaev2006}.   Our algorithm consists of two key steps.  First, in Sec.~\ref{subspace-section}, we compute an orthonormal basis for the span of the two nearly defective eigenvectors of $A_\eps$, using a shift-and-invert--like algorithm with some iterative refinement to circumvent conditioning problems for the second basis vector.  Second, in Sec.~\ref{chain-section}, we use this basis to project to a $2\times2$ matrix, find a defective matrix within $\mathcal{O}(\eps)$ of that, and use this nearby defective matrix to approximate the Jordan chain of $A_0$ to $\mathcal{O}(\eps)$. We have successfully applied this method to study EPs of large sparse matrices arising from discretizing the Maxwell equations of electromagnetism~\cite{joannopoulos2011photonic}, and include an example calculation in Sec.~\ref{numerical-section}.

In this sort of problem, we are \emph{given} the existence of a $2\times2$ Jordan block of an unknown but \emph{nearby} matrix $A_0$, and we want to find the Jordan chain. (Larger Jordan blocks, which rarely arise in practice for large matrices~\cite{Lin2016}, are discussed in Sec.~\ref{conclusions-section}.)  To uniquely specify the Jordan vector $j_0$ (to which we can add any multiple of $x_0$~\cite{mailybaev2006}), we adopt the normalization choice $\|x_0\|=1$ and $x_0^*j_0 = 0$. 
Our algorithm finds approximations with \emph{relative}   errors $\|\lambda_\eps - \lambda_0\|/\|\lambda_0\|$, $\|x_\eps - x_0\|/\|x_0\|$, and  $\|j_\eps - j_0\|/\|j_0\|$ which are  $\mathcal{O}(\eps)$. 
The nearly degenerate eigenvalues and eigenvectors of $A_\eps$ are given via perturbation theory~\cite{Moiseyev2011} by Puiseux series:
\begin{equation}
x^\pm = x_0 \pm c \eps^{1/2} j_0 + \eps w \pm \mathcal{O}(\eps^{3/2})
\label{eq:Puiseux}
\end{equation}
for some constant $c$ and a vector $w$, and similarly for the eigenvalues $\lambda^\pm$.  This means that a naive eigenvalue algorithm to find $x_\eps$ by simply computing an eigenvector of $A_\eps$ will only attain $\mathcal{O}(\eps^{1/2})$ accuracy, and furthermore that computing both $x^\pm$ eigenvectors is numerically problematic for small $\eps$ because they are nearly parallel.   In contrast, the invariant subspace \emph{spanned} by $x^\pm$ varies smoothly around $A_0$ [as can easily be seen by considering $x^+ + x^-$ and $\eps^{-1/2} (x^+ - x^-)$] and, with some care, an orthonormal basis for this subspace can be computed accurately by a variant of shifted inverse iteration (Sec.~\ref{subspace-section}). From that invariant subspace, we can then compute $x_\eps$ and so on to $\mathcal{O}(\eps)$ accuracy, which is optimal: because the set of defective matrices forms a continuous manifold, there are infinitely many defective matrices within $\mathcal{O}(\eps)$ of $A_\eps$, and hence we cannot determine $A_0$ to better accuracy without some additional information about the desired $A_0$, such as $dA/dp$.

An algorithm to solve this problem was proposed by Mailybaev~\cite{mailybaev2006}, which uses the SVD to approximate the perturbation required to shift $A_\eps$ onto the set of defective matrices, but the dense SVD is obviously impractical for large discretized PDEs and similar matrices.    Given $A_0$, Leung and Chang~\cite{leung1999null} suggest computing $x_0$ by either an SVD or a shifted inverse iteration, and then show that $j_0$ can be computed by solving an additional linear system (which is sparse for sparse $A$).  Those authors did not analyze the accuracy of their algorithm if it is applied to $A_\eps$ rather than $A_0$, but we can see from above that a naive eigensolver only computes $x_0$  to $\mathcal{O}(\eps^{1/2})$ accuracy, in which case they obtain similar accuracy for $j_0$ (since the  linear system they solve to find $j_0$ depends on $x_0$).   If an $\mathcal{O}(\eps)$ algorithm is employed to compute $x_0$ accurately (e.g. via our algorithm below), then Leung and Chang's algorithm computes $j_0$ to $\mathcal{O}(\eps)$ accuracy as well, but requires an additional linear solve compared to our algorithm.  It is important to note that our problem is very different from Wilkinson's problem~\cite{golub1976ill}, because in our case the Jordan structure (at least for the eigenvalue of interest) is known, making it possible to devise more efficient algorithms than for computing an unknown Jordan structure.  (The latter also typically involve dense SVDs, Schur factorizations, or similar~\cite{edelman1997geometric,lippert1999computation,kaagstrom1980algorithm}.)  Our algorithm differs from these previous works in that it relies primarily on inverse
iteration (and explicitly addresses the accuracy issues thereof), and thus is suitable for use with the large (typically sparse) matrices arising in 
physical applications.  Moreover we perform an analysis of the error in terms
of $\eps$ (i.e., the distance between $A_\eps$ and $A_0$), which was absent from previous works.

%% file: subspace.tex
In this section, we describe  two algorithms for computing the invariant subspace  of $A_\eps$ spanned by the eigenvectors whose eigenvalues are near  $\lambda_0$.   
Both algorithms begin by using   standard methods (e.g.,  Arnoldi or shifted inverse iteration~\cite{trefethen1997numerical}) to  find $u_1$ as an eigenvector of $A_\eps$. 
Then, we  find an additional  basis vector $u_2$ as  an  eigenvector of  $P^\perp A_\eps$, where $P^\perp\equiv\mathbb{I}-u_1u_1^*$ is the orthogonal projection onto the subspace perpendicular to $u_1$.  We denote   by  $U_\eps$ the matrix  whose columns  are the orthonormal basis vectors, $u_1$ and $u_2$.   
In Algorithm~\ref{subspace-dense-alg}, $u_2$ is found by performing inverse iteration  on the operator $P^\perp A_\eps$ (lines 3--8).  Since $P^\perp A_\eps$  does not preserve the sparsity pattern of  $A_\eps$, this algorithm is scalable only when using  matrix-free iterative solvers (that is, solvers which only require fast matrix-vector products).

\begin{algorithm}
  \caption{for finding the invariant subspace $U_\eps$ using matrix-free iterative solvers}
  \begin{algorithmic}[1]
    \State Find an eigenvector $u_1$ of $A_\eps$ with eigenvalue near $\mu$  (e.g., using  inverse iteration).  
    \State $P^\perp \gets \mathbb{I}-u_1u_1^*$
    \State $\lambda' \gets \mu$
    \State $u_2 \gets $ a random vector orthogonal to $u_1$.
    \While{$\|(P^\perp A_\eps  - \lambda'\mathbb{I}) u_2\| > \text{tol}$}
      \State $v \gets  (P^\perp A_\eps-\lambda'\mathbb{I})^{-1} u_2$
      \State $u_2 \gets v / \|v\|$
      \State $\lambda' \gets u_2^*P^\perp A_\eps u_2$
    \EndWhile
    \State Set $U_\eps = (u_1, u_2)$.
     \end{algorithmic}
 \label{subspace-dense-alg}
\end{algorithm}
Note that this iteration could easily be changed from Rayleigh-quotient inverse iteration to an Arnoldi Krylov-subspace procedure. In practical applications where the matrix parameters were already tuned to force a near EP, however, the estimate $\mu$ is close enough to the desired eigenvalue $\lambda_2$ that convergence only takes a few steps of Algorithm~\ref{subspace-dense-alg}.

Alternatively, when using sparse-direct solvers, one can implement  Algorithm~\ref{subspace-sparse-alg}, which performs inverse iteration on $A_\eps$ and only then applies $P^\perp$.
In order to see the equivalence of the two algorithms, let us  denote the nearly degenerate  eigenvectors of $A_\eps$ by $u_1$ and $x_2$ [i.e., $A_\eps u_1 = \lambda_1u_1$ and $A_\eps x_2 = \lambda_2 x_2$, with $\lambda_1\approx\lambda_2$]. 
While Algorithm~\ref{subspace-dense-alg} finds the eigenvector  $u_1$  and then computes  an    eigenvector of $P^\perp A_\eps$,  �Algorithm~~\ref{subspace-sparse-alg} computes  in the second step the orthogonal  projection  of the second eigenvector of $A_\eps$, $P^\perp x_2$. 
The equivalence of the Algorithms follows from the fact that   the orthogonal projection of $x_2$ is precisely  an eigenvector of $P^\perp A_\eps$ [since  
$(P^\perp A_\eps)(P^\perp x_2) = (P^\perp A_\eps)(\mathbb{I}-u_1u_1^*) x_2 = 
P^\perp (A_\eps x_2) - P^\perp (A_\eps u_1) u_1^*x_2 = \lambda_2 P^\perp x_2 
-\lambda_1 P^\perp u_1 u_1^*x_2=  \lambda_2 (P^\perp x_2)$].
Note, however, that  a subtlety arises when using Algorithm~\ref{subspace-sparse-alg} since  
the vector $(A_\eps-\mu\mathbb{I})^{-1}u_2$  (line 3) is nearly  parallel to $u_1$ and, consequently, the roundoff error in the projection $P^\perp$ is significant. To overcome this difficulty, we use an iterative refinement procedure~\cite{wilkinson1994rounding} (lines 5--7)  before  continuing the  inverse iteration.

\begin{algorithm}
  \caption{for finding the invariant subspace $U_\eps$ using sparse-direct solvers}
  \begin{algorithmic}[1]
    \State Do steps 1-4 in Algorithm 1.
    \While{$\|(P^\perp A_\eps  - \lambda'\mathbb{I} ) u_2\| > \text{tol}$}
      \State $v \gets P^\perp (A_\eps-\lambda'\mathbb{I})^{-1} u_2$
      \State $e \gets u_2 - P^\perp (A_\eps-\lambda'\mathbb{I}) v$
      \While{$\|e\| > \text{tol}$}
        \State $v \gets v + P^\perp (A_\eps-\lambda'\mathbb{I})^{-1} e$
        \State $e \gets u_2 - P^\perp (A_\eps-\lambda'\mathbb{I}) v$
      \EndWhile
      \State $u_2 \gets v / \|v\|$
      \State $\lambda' \gets u_2^*P^\perp A_\eps u_2$
    \EndWhile
    \State $U_\eps \gets (u_1, u_2)$.
     \end{algorithmic}
\label{subspace-sparse-alg}
\end{algorithm}

The importance of the invariant subspace $U_\eps$ is that it is, in fact, quite close
to the invariant subspace for the defective matrix $A_0$.  Thus it is possible
to find good approximants of the eigenvector and Jordan vector of $A_0$ within
the column space of $U_\eps$, as stated in the following lemma.

\begin{lemma}
Let $A_\eps$ be near to a  defective matrix $A_0$ (i.e., $\|A_\eps-A_0\| = \mathcal{O}(\eps)$), and $U_\eps$ be the invariant subspace of its nearly degenerate eigenvalue. Then 
  there exists a matrix $U_0$ which spans the invariant subspace for $A_0$
  such that $\|U_0-U_\eps\| = \mathcal{O}(\eps)$.
  \label{basis-close-lem}
\end{lemma}

This Lemma establishes that the invariant subspace of nearly degenerate eigenvalues varies smoothly at the vicinity of the EP, a property that has been previously established (e.g.,~\cite[Chapter~2.1.4]{kato2013perturbation}) and can be easily proven using the following proposition:
\begin{propn}
  Let $v_0$ be any vector in the invariant subspace of $A_0$. 
  Then there exists a vector $v$ in the column space of $U_\eps$ such that 
  $\|v-v_0\| = \mathcal{O}(\eps\|v_0\|)$.
\label{subspace-close-propn}
\end{propn}
\begin{proof}
  Expand $v_0$ in the basis $\{x_0,j_0\}$ consisting of the eigenvector
  and Jordan vector for $A_0$, so that 
  \[
    v_0 = \alpha x_0 + \beta j_0.
  \]
  According to perturbation theory near $A_0$, the eigenvectors 
  $x^{\pm}$ of $A_\eps$ can be expanded in a Puiseux series [\eqref{Puiseux}]. 
Let us choose $v = ax^+ + bx^-$ with $a=(\alpha+\beta\eps^{-1/2}/c)/2$
 and $b=(\alpha-\beta\eps^{-1/2}/c)/2$.  By construction, $v$ is in the 
 invariant subspace of $A_\eps$ and, therefore, also in the column space
  of $U_\eps$.  Moreover,
  \[
    \|v_0 - v\| = \eps \alpha w + \mathcal{O}[(|\alpha|+|\beta|\eps^{-1/2})\eps^{3/2}] 
    = \mathcal{O}(\eps\|v_0\|).
  \]
\end{proof}

Proposition~\ref{subspace-close-propn} allows us to approximate vectors in
$U_0$ (the invariant subspace of $A_0$) with vectors in $U_\eps$ (the invariant 
subspace in $A_\eps$).  To prove Lemma~\ref{basis-close-lem}, we need to show that the 
converse is also true, i.e., that  vectors in $U_0$ can be approximated by vectors in $U_\eps$.  Since $U_0$ and $U_\eps$  have equal dimensions, this follows as a corollary.

%% file: chain.tex
In this section we introduce vectors, $x_\eps$ and  $j_\eps$, in the column space of  $U_\eps$, which approximate the Jordan chain vectors, $x_0$ and $j_0$,
with $\mathcal{O}(\eps)$ accuracy.
Since  the first and second columns of  $U_\eps$ are  eigenvectors of $A_\eps$ and  $P^\perp A_\eps$ respectively,
a naive guess would be  to set $x_\eps = u_1$ and  $j_\eps\propto u_2$.  However, we know from perturbation theory that such an approach leads to relatively large errors since 
$\|u_1-x_0\| = O(\eps^{1/2})$, and we show below that we can obtain a more accurate approximation.
Algorithm~\ref{xj-alg} presents a prescription for constructing  $x_\eps$ and  $j_\eps$.

\begin{algorithm}
  \caption{for finding approximate Jordan chain vectors.}
  \label{xj-alg}
  \begin{algorithmic}[1]
  \State Introduce $S_\eps=U_\eps^*A_\eps U_\eps$ (denote entries of $S_\eps$  by $s_{ij}$).
    \State Set $\lambda_\eps = (s_{22}+s_{11})/2$.
    \State Set $\gamma = (s_{22}-s_{11})/2s_{12}$.
    \State Set $x_\eps = u_1 + \gamma u_2$ and $j_\eps = u_2/s_{12}$.
    \State Normalize $x_\eps$ and $j_\eps$.
  \end{algorithmic}
\end{algorithm}

In the remainder  of this section, we prove the following theorem:
\begin{thm}
Suppose  $A_0$ is  a defective matrix with  Jordan chain vectors $x_0$ and $j_0$,  
and  $A_\eps$ is  a nearby matrix such that $\|A_\eps - A_0\| = \mathcal{O}(\eps)$. 
The  vectors $x_\eps$ and $j_\eps$,   constructed via  Algorithms~\ref{subspace-dense-alg}--\ref{xj-alg},  satisfy
  $\|x_0-x_\eps\|, \|j_0-j_\eps\| = \mathcal{O}(\eps)$.
  \label{thm1}  
\end{thm}


To  prove Theorem~\ref{thm1}, we first introduce  the  $2\times2$ defective matrix
\begin{equation}
 \tilde{S} = \left(\begin{array}{cc}
    s_{11} & s_{12} \\ 
    -\frac{(s_{11}-s_{22})^2}{4s_{12}} & s_{22} 
  \end{array} \right).
  \label{eq:tildeS-defn}
\end{equation}
The vectors $x_\eps$ and $j_\eps$, constructed above, are related to the Jordan chain vectors $x_\varepsilon'$ and $j_\varepsilon'$ of $\tilde{S}$ via  $x_\eps = U_\eps x'_\eps$ and $j_\eps = U_\eps j'_\eps$. [More explicitly, one can verify that 
the   vectors $x'_\eps = [1 ; \gamma]$ and $j'_\eps = [0 ; s_{12}^{-1}]$ satisfy the chain relations:
$\tilde{S}x'_\eps = \lambda_\eps x'_\eps$ and $\tilde{S}j'_\eps = \lambda_\eps j'_\eps + x'_\eps$.]
Using this observation, we prove Theorem~\ref{thm1}  in two steps: First, we    introduce the $2\times2$ defective matrix $S_0 \equiv  U_0^*A_0U_0$, and  show that  the    Jordan chain vectors  of $S_0$ and $\tilde{S}$ are within $\mathcal{O}(\eps)$ of each other (Lemma~\ref{lemma2}). 
Then, we use Lemma~\ref{lemma3} to show  that the proximity of  the  Jordan chains of $S_0$ and $\tilde{S}$ implies the proximity of $\{x_\eps, j_\eps\}$ and the Jordan chain vectors of $A_0$. In the remainder of this section, we prove Lemmas~\ref{lemma2}--\ref{lemma3}.

\begin{lemma}
\label{lemma2}
The  Jordan chains  of $S_0$ and $\tilde{S}$ are  within $\mathcal{O}(\eps)$ of  each other.  
\end{lemma}

\begin{proof}
Any defective matrix is similar to a  $2\times2$ defective matrix of the form
  \begin{equation}
   S =  \left(\begin{array}{cc}
      \lambda + \alpha\beta & \alpha^2 \\
      -\beta^2 & \lambda - \alpha\beta
    \end{array}\right),
    \label{eq:GeneralDefect}
  \end{equation}
with  appropriate complex parameters $\alpha,\beta,\lambda$~\cite{edelman1997geometric}.   
The   Jordan chain vectors of $S$ are
  \begin{equation}
    x' = \frac{1}{\sqrt{|\alpha|^2+|\beta|^2}}
    \left(\begin{array}{c}\alpha\\-\beta\end{array}\right),
    \quad
    j' = \frac{1}{|\alpha|^2+|\beta|^2}
    \left(\begin{array}{c}\beta\\\alpha\end{array}\right).
    \label{eq:jordan-chain-eq}
  \end{equation}
 
 The matrix $\tilde{S}$ can be recast in the form of \eqref{GeneralDefect} with 
  $\tilde{\alpha}^2 = s_{12}$, 
  $\tilde{\beta}^2=\frac{(s_{22}-s_{11})^2}{4s_{12}}$, and 
  $\tilde{\lambda} = (s_{11}+s_{22})/2$.  
To rewrite  $S_0$ in the form of \eqref{GeneralDefect},  we introduce  $\alpha_0, \beta_0$ and $\lambda_0$. 
From  \eqref{jordan-chain-eq}, in order to prove the proximity of the Jordan chains,  it remains to show that 
 $|\tilde{\alpha}-\alpha_0|$ 
  and $|\tilde{\beta}-\beta_0|$ are both $\mathcal{O}(\eps)$. 

Since $u_1$ is an eigenvector of $A_\eps$ and is also  orthogonal to $u_2$, $S_\eps$ is upper triangular (i.e. $s_{21} = 0$). By construction, 
the matrices $\tilde{S}$ [\eqref{tildeS-defn}] and $S_\eps\equiv U_\eps^*A_\eps U_\eps$ differ only in the lower-left entry, which has the form $(s_{11}-s_{22})^2 / 4s_{12}$.  Since $s_{11}$ and $s_{22}$ are the nearly degenerate eigenvalues of $A_\eps$, they differ by $\eps^{1/2}$, so this entry is of size  $\mathcal{O}(\eps)$, which implies $\|\tilde{S}-S_\eps\|=\mathcal{O}(\eps)$.
From  Lemma~\ref{basis-close-lem} we have $\|U_\eps-U_0\|=\mathcal{O}(\eps)$, which implies that $\|S_\eps-S_0\|=\mathcal{O}(\eps)$, and it follows that $\|\tilde{S}-S_0\| = \mathcal{O}(\eps)$. Then, using  \eqref{GeneralDefect}, we conclude that $|\tilde{\alpha}^2-\alpha_0^2| = \mathcal{O}(\eps)$, so that either $|\tilde{\alpha}-\alpha_0|=\mathcal{O}(\eps)$ or $|\tilde{\alpha}+\alpha_0|=\mathcal{O}(\eps)$ and we may choose the sign of $\alpha_0$ so that $|\tilde{\alpha}-\alpha_0|=\mathcal{O}(\eps)$.  To prove that $|\tilde{\beta}-\beta_0| = \mathcal{O}(\eps)$, we first bound
  \[
    2|\tilde{\alpha}\tilde{\beta}-\alpha_0\beta_0|
    \leq 
    |(\tilde{\alpha}\tilde{\beta}+\tilde{\lambda})-(\alpha_0\beta_0+\lambda_0)|
    + |(\tilde{\alpha}\tilde{\beta}-\tilde{\lambda}) 
    - (\alpha_0\beta_0 - \lambda_0)| = \mathcal{O}(\eps),
  \]
  where the two terms on the right are bounded by $\mathcal{O}(\eps)$ as they   are terms contributing to $\|\tilde{S}-S_0\|$.   Now since    $|\tilde{\alpha}-\alpha_0|=\mathcal{O}(\eps)$,  the above inequality implies that $|\tilde{\beta}-\beta_0|=\mathcal{O}(\eps)$.  
\end{proof}

\begin{lemma}
\label{lemma3}
 Given that the  Jordan chains of $S_0$ and $\tilde{S}$ are within $\mathcal{O}(\eps)$ of   each other,
it follows that the Jordan chain of $A_0$ and the vectors $\{x_\eps,j_\eps\}$ are within $\mathcal{O}(\eps)$ of  each other.
 \end{lemma}
\begin{proof}
The Jordan chain vectors of the $2\times2$ matrix  $S_0$  are related to the Jordan chain vectors of the large matrix   $A_0$ via:
   $x_0 = U_0  x_0'$,  $j_0 =  U_0  j_0'$.  
Using  the standard triangle inequality, we have
$\|x_0 - x_\eps \| = \|U_0 x'_0 - U_\eps x'_\eps\| 
\leq \|U_0 x'_0 - U_0 x'_\eps\| + \|U_0 x'_\eps - U_\eps x'_\eps\|$.
It follows that
$\|x_0 - x_\eps\| \leq \|U_0\|\|x'_0 - x'_\eps\| + \|U_0 - U_\eps\|\|x_0\|.$
 Both terms on the right-hand side are quantities that we already know are $\mathcal{O}(\eps)$.  The same argument can be used to prove that $\|j'_0 - j'_\eps\|=\mathcal{O}(\eps)$.
\end{proof}



%% file: numerical_results.tex
In this section, we analyze the accuracy of our shift-and-invert-based algorithm and compare it with  the SVD-based algorithm presented  by Mailybaev~\cite{mailybaev2006}.
We apply  both algorithms to dense defective $50\times50$ matrices, which are randomly perturbed by an
amount $\eps$.   Figure~\ref{fig1} shows the relative  errors in eigenvalues and Jordan vectors  as a function of $\eps$. 
We find that both methods   are $\mathcal{O}(\eps)$ accurate. However, our method is scalable for large defective matrices (as demonstrated below), whereas the SVD-based method becomes impractical. 

When the derivative of the matrix $dA/dp$ is known, the accuracy of both algorithms can be improved. 
This can happen, for example,  when $A$ is an operator arising from a physical problem which depends on  $p$ in a known way. An extension of the SVD-based algorithm to incorporate the knowledge of $dA/dp$ is given in~\cite{mailybaev2006}. 
Our algorithm can be improved as follows.  First, we employ the adjoint method to find the value
$p$ for which $\mathcal{A}_\eps \equiv A_\eps + p\,dA/dp = A_0 + O(\eps^2)$. More explicitly, we   compute the derivative $dg/dp$, where $g$ is a function constructed to be equal to 0 at the exceptional point, and  then  take  a single Newton step in $p$ to  obtain $\mathcal{A}_\eps$. More details are given in the appendix, Sec.~\ref{adjoint-section}, (where it is assumed that the matrices $A_0$ and $A_\eps$ are real, but the resulting formula works for both real and complex matrices).
 Then, by applying Algorithms~\ref{subspace-dense-alg}--\ref{xj-alg} to $\mathcal{A}_\eps$, we obtain a Jordan vector to accuracy $\eps^2$, with the additional cost of a single linear solve compared to the first-order method.  Therefore, we refer to the modified algorithm as a second-order method.

\begin{figure}[h!]
  \begin{center}
    \includegraphics[scale=0.6]{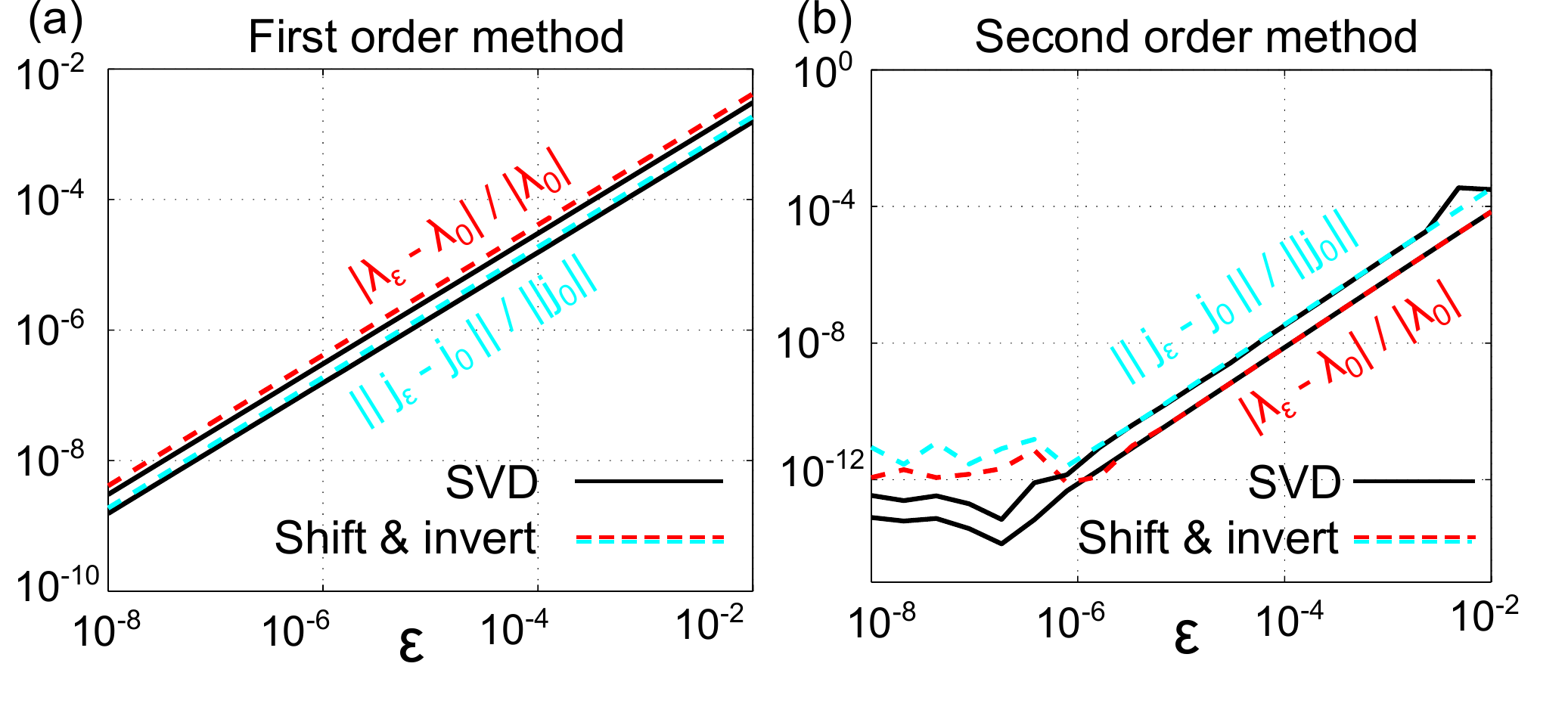}
  \end{center}
  \caption{\textbf{Relative errors in eigenvalues $\mathbf{(|\lambda_\eps-\lambda_0|/|\lambda_0|)}$ and Jordan vectors $\mathbf{(\|j_\eps-j_0|/|j_0\|)}$.}
  (a) Comparison of our shift-and-invert based first-order method (red and cyan dashed lines) and the previous SVD-based algorithm~\cite{mailybaev2006} (black lines).
     All algorithms demonstrate $\mathcal{O}(\eps)$ convergence.
  (b) Comparison of our second-order algorithm, which uses the adjoint method (red and cyan dashed lines), and the second-order SVD-based algorithm (black lines). Both methods demonstrate $\mathcal{O}(\eps^2)$ convergence.}
  \label{fig1}
\end{figure}

The accuracy of our second-order method and a comparison with the SVD-based second-order method are presented in Fig.~1b. Both methods show $\mathcal{O}(\eps^2)$ convergence. Note that a floor of $10^{-12}$ in the relative error is reached due to rounding errors in the construction of the exactly defective matrix $A_0$.

In the remainder of this section, we show an application of our algorithm to  a physical problem of  current interest: Computing spontaneous emission 
rates from  fluorescent molecules which are placed near electromagnetic resonators with  EPs~\cite{Pick2016}. 
The resonant modes of a  system can be  found by solving Maxwell's equations, which can be written in the form of an    eigenvalue problem: $\hat{A}x_n=\lambda_n x_n$~\cite{joannopoulos2011photonic}. Here, Maxwell's  operator is
$\hat{A}\equiv\varepsilon^{-1}\nabla\times\nabla\times\,$ (where $\varepsilon$ is the dielectric permittivity of the medium), $x_n$ are the eigenmodes (field profiles) and the eigenvalues  $\lambda_n\equiv\omega_n^2$ are the squares of the resonant frequencies. 
We  discretize Maxwell operator, $\hat{A}$, using finite differences~\cite{christ1987three,Pick2016} to a matrix $A$.
As recently shown in~\cite{Pick2016}, the spontaneous emission   rate at an EP  can be  accurately computed using a simple formula that includes  the degenerate mode $x_n$, Jordan vector $j_n$,  and eigenvalue $\lambda_n$.  
To demonstrate our algorithm, we computed $x_n, j_n$ and $\lambda_n$ for the numerical example of two coupled plasmonic resonators   (Fig.~\ref{Plasmonics}a). The system consists  of two rectangular silver (Ag) rods covered by a  silica (SiO$_2$) coating, with commensurate distributions of gain and loss in the coating.  
When adding gain and loss, the eigenvalues move in the complex plane  (blue and red lines in Fig.~\ref{Plasmonics}b) and, at a critical amount of gain and loss, they  merge at an  EP (orange dot).  
Fig.~\ref{Plasmonics}c presents approximate   eigenvector $x_\eps$ and  Jordan vector $j_\eps$, computed via our shift-and-invert-based algorithm,
and   normalized so that $x_0^*j_0 = 1$ and $j_0^*j_0 = 0$. (Only the real parts of the vectors are shown in the figure.) 
The matrix $A_\eps$ in this example is  $212^2\times212^2$. An SVD-based approach  would require enormous computational  resources for this relatively small   two-dimensional problem, whereas our algorithm required only a few seconds on a laptop. 
The result was validated in~\cite{Pick2016} by comparing the Jordan-vector perturbation-theory predictions to explicit eigenvector calculations near the EP.

\begin{figure}[h]
\centering
\includegraphics[width=1\textwidth]{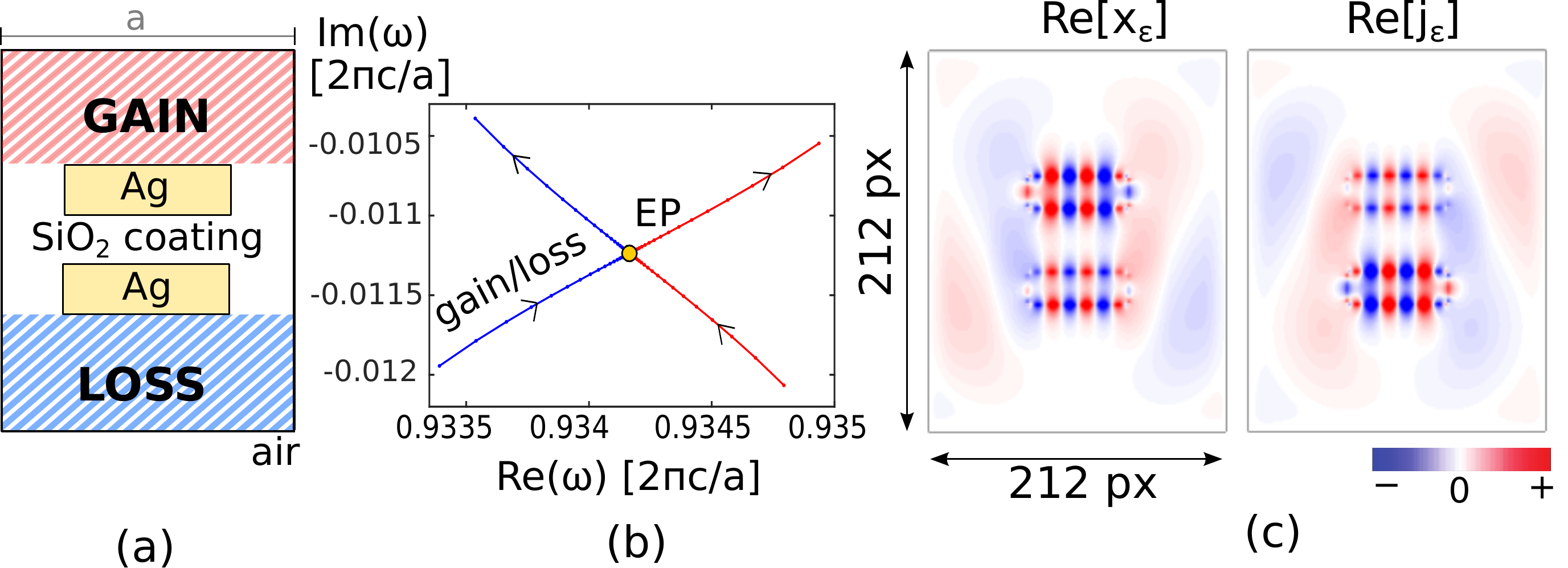}
\caption{\textbf{Application of the algorithm for problems in electromagnetism.}
(a) Schematics of the structure. Two silver rods (Ag) covered by  a silica coating (SiO$_2$).  Gain and loss  are added to the outer sides of the coating. 
(b) By adding gain and loss, the eigenvalues move in the complex plane (blue and red curves) and merge at an EP (orange dot).
(c) The approximate degenerate eigenvector $x_\eps$ and Jordan vector $j_\eps$.
(a) and (b) are borrowed from~\cite{Pick2016}.}
\label{Plasmonics}
\end{figure}


%% file: adjoint.tex
In this section we suppose we know how to compute the derivaive of the matrix $dA/dp$ and we explain how to use the adjoint method~\cite{errico1997adjoint,strang2007computational} in order to find the value
$p$ for which $A_\eps + pdA/dp = A_0 + O(\eps^2)$.  
The derivation in this section assumes that all matrices are real but,
with slight generalizations (some complex conjugations) summarized in Algorithm~\ref{complex-adjoint-alg},
the formula we provide in the end works both for real and complex matrices.

To find the exceptional point, we consider the matrices $U$ and $S$ as 
depending on the parameter $p$.  To do this, we would like to define
$U(p)$ and $S(p)$ by the equations
\begin{align*}
  &A(p)U(p) = U(p)S(p) \\
  &U(p)^TU(p) = Id_2,
\end{align*}
but these do not have a unique solution because of the possibility of
rotating the basis $U$.  When $U$ is complex, there are also phase degrees
of freedom.  To fix this, we can enforce that the first column of $U$, $u_1$,
is orthogonal to some vector $w$.  We will actually choose $w$ such that when
$p=0$, the matrix $S$ is upper triangular.  Thus $U(p)$ and $S(p)$ solve
$f(U,S,p) = 0$, where
\[
  f(U,S,p) = \left( \begin{array}{c}
    Au_1 - s_{11} u_1 - s_{21} u_2 \\
    Au_2 - s_{12} u_1 - s_{22} u_2 \\
    u_1^Tu_1 - 1 \\
    u_1^Tu_2 \\
    u_2^Tu_2 - 1 \\
    u_1^Tw\end{array}\right).
\]
Now we would like to find the value of $p$ such that $S(p)$ is defective and therefore satisfies
\[
  g(S) = (\trace(S)/2)^2 - \det(S) = 0.
\]
By computing $g(0)$ and $dg/dp$, we can find the correct value
of $p$ at which $g(p)=0$ and we accomplish this task by using the adjoint method, as explained below.

Using the chain rule, we have
\begin{equation}
dg/dp = (\partial g/\partial S) (dS/dp).
\label{dgdp-eq}
\end{equation}
The derivative $dS/dp$ can be found from differentiating $f(U,S,p)$; it
satisfies
\[
  \partial f/\partial p + (\partial f/\partial U) dU/dp + 
  (\partial f/\partial S) dS/dp = 0.
\]
Combining the unknowns into a single variable
$X=(u_1,u_2,s_{11},s_{12},s_{21},s_{22})$, the equation simplifies to
\[
  \partial f/\partial p + (\partial f/\partial X) dX/dp = 0.
\]
Substituting  this into~Eq.~\ref{dgdp-eq}, we obtain
\[
  dg/dp = -(\partial g/\partial X) (\partial f/\partial X)^{-1} 
  (\partial f/\partial p).
\]
To compute this, we let 
$\Lambda^T = (\partial g/\partial X) (\partial f/\partial X)^{-1}$, so that
\[
  (\partial f/\partial X)^T \Lambda = (\partial g/\partial X).
\]
The matrix $\partial f/\partial X$ takes the form
\[ \partial f/\partial X = \left(\begin{array}{cccccc}
    A-s_{11} & -s_{21} & -u_1 & 0 & -u_2 & 0 \\
    -s_{12} & A-s_{22} & 0 & -u_1 & 0 & -u_2 \\
    2u_1^T & 0 & 0 & 0 & 0 & 0 \\
    u_2^T & u_1^T & 0 & 0 & 0 & 0 \\
    0 & 2u_2^T & 0 & 0 & 0 & 0 \\
    0 & 0 & 0 & 0 & 1 & 0 
  \end{array}\right).
\]
Moreover the vector $\partial g/\partial X$ has the form
$\partial g/\partial X = (0, 0, h_{11}, h_{12}, h_{21}, h_{22})^T$
(the zeros reflecting the fact that $g$ is independent of $U$), where
\begin{align}
  \partial g/\partial S = 
  \left(\begin{array}{cc}
    h_{11} & h_{12} \\
    h_{21} & h_{22} \end{array}\right) = 
  \left(\begin{array}{cc}
    (s_{11}-s_{22})/2 & s_{21} \\
    s_{12} & (s_{22}-s_{11})/2
  \end{array}\right).
\label{h-definition}
\end{align}
Expanding the adjoint variables as 
$\Lambda = (\lambda_1^T,\lambda_2^T,\sigma_{11},\sigma_{12},\sigma_{22},
\sigma_{21})^T$, we
obtain the adjoint equations
\begin{align}
  (A^T-s_{11})\lambda_1 &=s_{12}\lambda_2-2\sigma_{11}u_1-\sigma_{12}u_2
  \label{lambda1-eq}
  \\(A^T-s_{22})\lambda_2 &= -\sigma_{12}u_1 - 2\sigma_{22}u_2,
  \label{lambda2-eq}
\end{align}
and the normalization equations
\begin{align}
  -u_1^T\lambda_1 &= h_{11} 
  \label{u1h11-eq}\\
  -u_1^T\lambda_2 &= h_{12} 
  \label{u1h12-eq}\\
  -u_2^T\lambda_2 &= h_{22}.
  \label{u2h22-eq}
\end{align}

To solve this system, we first take the dot product of~Eq.~\ref{lambda1-eq}
with $u_1$, resulting in
\begin{align*}
g_1^Tu_1 + s_{12}\lambda_2^Tu_1 - 2\sigma_{11} = 0.
\end{align*}
Now applying~Eq.~\ref{u1h12-eq}, this simplifies to 
\begin{equation}
  \sigma_{11} = \frac{-h_{12}s_{12}}{2}.
\end{equation}
Similarly, taking the dot product of~Eq.~\ref{lambda2-eq} with $u_2$
and using~Eq.~\ref{u2h22-eq}, we find that 
\begin{equation}
  \sigma_{22} = \frac{h_{12}s_{12}}{2}.
\end{equation}
To find $\sigma_{21}$, we take the dot product of Eq.~\ref{lambda2-eq} 
equation with $u_1$ and apply~Eq.~\ref{u1h11-eq},
\begin{align*}
  -(s_{11}-s_{22})h_{12} = 
  (s_{11}-s_{22})\lambda_2^Tu_1 =
  \lambda_2^T (A-s_{22})u_1 =
  - \sigma_{12}.
\end{align*}
This yields
\begin{align}
  \sigma_{12} = (s_{11}-s_{22})h_{12}.
\end{align}

Now that $\sigma_{ij}$ are known, we can solve for $\lambda_2$ and then 
for $\lambda_1$ using the first two equations.  Then we must add back some
multiple of the left eigenvectors of $A^T$ to $\lambda_2$ and $\lambda_1$ 
so that they satisfy the other two normalization conditions.
These steps are summarized in Algorithm~\ref{real-adjoint-alg}.
\begin{algorithm}
  \caption{The adjoint method for finding $p$}
  \label{real-adjoint-alg}
  \begin{algorithmic}[1]
    \State $\sigma_{12}\gets -(s_{11}-s_{22})h_{12}$ ($h$ is defined in Eq.~\ref{h-definition})
    \State $\sigma_{22} \gets -s_{12}h_{12}$
    \State Solve $(A-s_{22})^T\lambda_2  = \sigma_{12} u_1 - \sigma_{22} u_2$.
    \State Add a multiple of the left eigenvector of $A$ to $\lambda_2$
    so that $u_2^T\lambda_2 = -h_{22}$.
    \State Solve $(A-s_{11})^T\lambda_1 = s_{12}\lambda_2 - \sigma_{22} u_1
    - \sigma_{12} u_2$.
    \State Add a multiple of the left eigenvector of $A$ to $\lambda_1$
    so that $u_2^T\lambda_1 = -h_{21}$.
    \State Compute $dg/dp = - \lambda_1^T(dA/dp)u_1 - \lambda_2^T(dA/dp)u_2$.
    \State Set $p = -g(0)/(dg/dp)$.
  \end{algorithmic}
\end{algorithm}

\newpage

To derive the adjoint method for complex $A$, one can split the problem into
real and imaginary parts.  The resulting computation is described in 
Algorithm~\ref{complex-adjoint-alg}.
\begin{algorithm}
  \caption{The adjoint method with complex $A$}
  \begin{algorithmic}[1]
    \State Solve $2\alpha_1 - s_{12}\beta_4 i  = -s_{12}(g_S)_{12}$
    for real $\alpha_1$ and $\beta_4$.
    \State Set $a = -\overline{s_{11}-s_{22}}(-(g_S)_{12} + \beta_4 i)$.
    \State Set $\overline{b} = (s_{12}(g_S)_{11}-(g_S)_{22} -
    \overline{a}$.
    \State Solve $2\alpha_4 - \beta_3 i = s_{12}(g_S)_{12} + i\beta_4$
    for real $\alpha_4$ and $\beta_3$.
    \State Solve $(A-\overline{s_{22}})^* \lambda_2 = 
    -au_1 - (2\alpha_4 + \beta_3 i)u_2$.
    \State Add a left eigenvector of $A$ to $\lambda_2$ so that 
    $u_2^*\lambda_2 = -(g_S)_{22}$.
    \State Solve $(A-\overline{s_{11}})^* \lambda_1 = 
    s_{12}\lambda_2 - 2\alpha_1u_1 - (\overline{a+b})u_2$.
    \State Add a left eigenvector of $A$ to $\lambda_1$ so that 
    $u_2^*\lambda_1 = -(g_S)_{21}$
    \State Compute $dg/dp = -\lambda_1^*(dA/dp)u_1 -\lambda_2^*(dA/dp)u_2$.
    \State Set $p = -g(0) / (dg/dp)$.
  \end{algorithmic}
  \label{complex-adjoint-alg}
\end{algorithm}
